\documentclass[11pt,reqno]{amsart} %
\usepackage{amsmath,amscd}
\usepackage{amsthm}
\usepackage{amssymb}
\usepackage{amsfonts}
\usepackage{latexsym}
\usepackage{url}
\usepackage{latexsym}
\usepackage{graphicx,color,subfig}
\usepackage{enumerate,fancybox}
\usepackage{hyperref}

\DeclareMathOperator\dist{dist}

\newtheorem{theorem}{Theorem}[section]

\newtheorem{lemma}{Lemma}[section]
\newtheorem{definition}{Definition}[section]

\newtheorem{remark}{Remark}[section]

\newtheorem{corollary}{Corollary}[section]

\newtheorem*{mythm}{Theorem}
\newtheorem*{mylemma}{Lemma}

\begin{document}
\title[Strong asymptotics for Bergman polynomials]
{Strong asymptotics for Bergman polynomials over domains with corners}

\date{\today}

\author[N. Stylianopoulos]{Nikos Stylianopoulos}
\address{Department of Mathematics and Statistics,
         University of Cyprus, P.O. Box 20537, 1678 Nicosia, Cyprus}
\email{nikos@ucy.ac.cy}
\urladdr{http://ucy.ac.cy/~nikos}

\keywords{Bergman orthogonal polynomials, Faber polynomials, strong asymptotics, polynomial estimates, quasiconformal mapping, conformal mapping}
\subjclass[2000]{30C10, 30C30, 30C50, 30C62, 41A10}

\begin{abstract}
Let $G$ be a bounded simply-connected domain in the complex plane $\mathbb{C}$,
whose boundary $\Gamma:=\partial G$ is a Jordan curve, and let $\{p_n\}_{n=0}^{\infty}$ denote the sequence of Bergman polynomials of $G$. This is defined as the sequence
$$
p_n(z) = \lambda_n z^n+\cdots, \quad \lambda_n>0,\quad n=0,1,2,\ldots,
$$
of polynomials that are orthonormal with respect to the inner product
$$
\langle f,g\rangle := \int_G f(z) \overline{g(z)} dA(z),
$$
where $dA$ stands for the area measure.

The aim of the paper is to establish the strong asymptotics for $p_n$ and $\lambda_n$, $n\in\mathbb{N}$, under the assumption that $\Gamma$ is piecewise analytic. This complements an investigation started in 1923 by T.\ Carleman, who derived the strong asymptotics for domains with analytic boundaries and carried over by P.K.\ Suetin in the 1960's, who established them for domains with smooth boundaries.
\end{abstract}

\maketitle
\allowdisplaybreaks
\section{Introduction and main results}\label{section:intro}
Let $G$ be a bounded simply-connected domain in the complex plane $\mathbb{C}$,
whose boundary $\Gamma:=\partial G$ is a Jordan curve and let
$\{p_n\}_{n=0}^{\infty}$ denote the sequence of  Bergman polynomials of
$G$. This is defined as the sequence of polynomials
\begin{equation}\label{eq:pndef}
p_n(z) = \lambda_n z^n+ \cdots, \quad \lambda_n>0,\quad n=0,1,2,\ldots,
\end{equation}
that are orthonormal with respect to the inner product
$$
\langle f,g\rangle := \int_G f(z) \overline{g(z)} dA(z),
$$
where $dA$ stands for the area  measure. (As usual, we set $\|f\|_{L^2(G)}:=\langle f,f\rangle^{1/2}$.)

Let $\Omega:=\overline{\mathbb{C}}\setminus\overline{G}$ denote the complement of $\overline{G}$ and let $\Phi$ denote the conformal map  $\Omega\to\Delta:=\{w:|w|>1\}$, normalized so that near infinity
\begin{equation}\label{eq:Phi}
\Phi(z)=\gamma z+\gamma_0+\frac{\gamma_1}{z}+\frac{\gamma_2}{z^2}+\cdots,\quad \gamma>0.
\end{equation}
Finally, let $\Psi:=\Phi^{-1}:\Delta\to\Omega$ denote the inverse conformal map. Then,
\begin{equation}\label{eq:Psi}
\Psi(w)=bw+b_0+\frac{b_1}{w}+\frac{b_2}{w^2}+\cdots, \quad |w|> 1,
\end{equation}
where $b=1/\gamma$ gives the (\textit{logarithmic}) \textit{capacity} $\textup{cap}(\Gamma)$ of $\Gamma$.

The main purpose of the paper is to establish the strong asymptotics of the leading coefficients $\{\lambda_n\}_{n\in\mathbb{N}}$ and the Bergman polynomials $\{p_n\}_{n\in\mathbb{N}}$, in $\Omega$, for non-smooth boundary $\Gamma$. We do this under the assumption that $\Gamma$ is piecewise analytic without cusps, i.e., without zero or $2\pi$ angles. Thus, we allow $\Gamma$ to have corners. In this sense, our results complement an investigation started by T.~Carleman \cite{Ca23} in 1923, who derived the strong asymptotics under the assumption that $\Gamma$ is analytic, and was carried over by P.K.~Suetin \cite{Su74} in the 1960's, who verified them for smooth $\Gamma$.  As it turns out, the techniques employed in both \cite{Ca23} and \cite{Su74} are tied to the specific properties that characterize the mapping functions $\Phi$ and $\Psi$ when $\Gamma$ is analytic, or smooth, and therefore they are not suitable to treat domains with corner. To overcome this, we develop what we believe to be a novel approach. This approach involves, in particular, new techniques from the theory of quasiconformal mapping and a new sharp estimate concerning the growth of polynomials outside domains with corners.

Our main results are the following three theorems.
\begin{theorem}\label{thm:finelambdan}
Assume that the boundary $\Gamma$ of $G$ is piecewise analytic without cusps. Then, for any $n\in\mathbb{N}$,
\begin{equation}\label{eqinthm:finelambdan}
{\frac{n+1}{\pi}\frac{\gamma^{2(n+1)}}{\lambda_n^2}=1-\alpha_n,}
\end{equation}
where
\begin{equation}\label{eqinthm:finelambdanii}
0\le\alpha_n\le c_1(\Gamma)\,\frac{1}{n}.
\end{equation}
\end{theorem}

\begin{theorem}\label{thm:finepn}
Under the assumptions of Theorem \ref{thm:finelambdan}, for any $n\in\mathbb{N}$,
\begin{equation}\label{eqinthm:finepn}
{p_n(z)=\sqrt{\frac{n+1}{\pi}}\,\Phi^n(z)\Phi^\prime(z)
\left\{1+A_n(z)\right\}},\quad z\in\Omega,
\end{equation}
where
\begin{eqnarray} \label{eqinthm:finepnii1}
|A_n(z)|\le \frac{c_2(\Gamma)}{\dist(z,\Gamma)\,|\Phi^\prime(z)|}\,\frac{1}{\sqrt{n}}
+c_2(\Gamma)\,\frac{1}{n}.
\end{eqnarray}
\end{theorem}
Above and in the sequel we use $c(\Gamma)$, $c_1(\Gamma)$, $c_2(\Gamma)$, e.t.c., to denote non-negative  constants that depend only on $\Gamma$. We also use $\dist(z,B)$ to denote the (Euclidian) distance of $z$ from a set $B$ and call the quantities $\alpha_n$ and $A_n(z)$, defined by (\ref{eqinthm:finelambdan}) and (\ref{eqinthm:finepn}), as the \textit{strong asymptotic errors} associated with $\lambda_n$ and $p_n(z)$, respectively.

From (\ref{eqinthm:finepnii1}) and the well-known distortion  property of conformal mappings
\begin{equation}\label{eq:distortion}
\dist(\Phi(z),\partial\mathbb{D})\le 4\,\dist(z,\Gamma)\,|\Phi^\prime(z)|,\quad z\in\Omega,
\end{equation}
see e.g.\ \cite[p.~23]{ABbook}, we arrive at another estimate of $A_n(z)$, which does not involve the derivative of $\Phi$:
\begin{eqnarray} \label{eqinthm:finepnii2}
|A_n(z)|\le\frac{c_3(\Gamma)}{|\Phi(z)|-1}\,\frac{1}{\sqrt{n}}
+c_2(\Gamma)\,\frac{1}{n}, \quad z\in\Omega.
\end{eqnarray}

A partial answer regarding the sharpness of the exponent $1$ of $n$ in (\ref{eqinthm:finelambdanii}) is provided by the next theorem. This theorem is established under the assumption that $\Gamma$ belongs to a broader class of Jordan curves than the one appearing in Theorem~\ref{thm:finelambdan}, namely the class of quasiconformal curves. We recall that a Jordan curve $\Gamma$ is \textit{quasiconformal} if there is a constant $M$ such that,
$$
\textup{diam}\Gamma(z,\zeta)\le M |z-\zeta|,\mbox{ for all }\, z,\zeta\in\Gamma,
$$
where
$\Gamma(z,\zeta)$ is the arc (of smaller diameter) of $\Gamma$ between $z$ and $\zeta$.
In connection with the assumptions of Theorem~\ref{thm:finelambdan}, we also recall that a piecewise analytic Jordan curve is quasiconformal if and only if has no cusps. 
The assumption that $\Gamma$ is quasiconformal ensures the existence of an associated $K$-quasiconformal reflection $y(z)$, for some $K\ge 1$, characterized by the properties (A1)--(A4) in Section~\ref{sub:QuasiRef}. All our estimates derived under this assumption are given in terms of the constant
\begin{equation}\label{eq:refcoef}
k:=(K-1)/(K+1),
\end{equation}
which in the sequel we refer to  as a \textit{reflection factor of $\Gamma$}. We note that $0\le k<1$, with $k=0$ if $\Gamma$ is a circle.

In the next theorem we require that  $\Gamma$ is quasiconformal and rectifiable. (Note that there are examples of non-rectifiable quasiconformal curves. 
However, any quasiconformal curve has zero area.)
Our result shows that the strong asymptotic error $\alpha_n$ cannot decay faster than $(n+1)|b_{n+1}|^2$, where $b_{n+1}$ is the coefficient of $1/w^{n+1}$ in the Laurent series expansion (\ref{eq:Psi}) of $\Psi(w)$.
\begin{theorem}\label{thm:alphange}
Assume that $\Gamma$ is quasiconformal and rectifiable. Then, for any $n\in\mathbb{N}$,
\begin{equation}\label{eq:alphange}
\alpha_{n}\ge\,\frac{\pi\,(1-k^2)}{A(G)}\,(n+1)\,|b_{n+1}|^2,
\end{equation}
where $A(G)$ denotes the area of $G$ and $k$ is a reflection factor of $\Gamma$.
\end{theorem}
Theorem~\ref{thm:alphange} provides, in addition, a link between the problems of estimating the error in the strong asymptotics for $\lambda_n$ and that of estimating coefficients in the well-known class $\Sigma$, of functions analytic and univalent in $\Delta\setminus{\infty}$, having a Laurent series expansion of the form (\ref{eq:Psi}) with $b=1$. This latter is one of the best-studied problems in Geometric Function Theory.

In view of Theorem~\ref{thm:alphange}, to show that the estimate (\ref{eqinthm:finelambdanii}) is sharp, it is sufficient to find an example where $b_n\ge c_4/n$, for some large $n$. This is tricky, because if the Faber operator of $G$ is bounded, as it would be when $\Gamma$ is, for instance, piecewise analytic (even with cusps), then $b_n\le c_5/n$, for any $n\in\mathbb{N}$, see \cite{Ga99}. Here, we offer two examples supporting the hypothesis that $O(1/n)$ in (\ref{eqinthm:finelambdanii}) cannot be improved.

The first example is based on a Jordan curve constructed by Clunie in \cite{Cl59}, for which the sequence $\{n\,b_n\}_{n\in\mathbb{N}}$ is unbounded. More precisely, for some $\epsilon>0$ and some subsequence $\mathcal{N}$ of $\mathbb{N}$,
\begin{equation}\label{eq:clunie}
n|b_n|\ge n^\epsilon,\quad n\in\mathcal{N}.
\end{equation}
It was shown by Gaier in \cite[\S~4.2]{Ga99} that Clunie's curve is, eventually, quasiconformal.

The second example is generated by the function $\Psi(w)=w+\frac{1}{nw^n}$. For any $n\ge 2$, this $\Psi$ maps $\Delta$ conformally onto the exterior of a $(n+1)$-cusped hypocycloid $Y_{n+1}$, which is a piecewise analytic Jordan curve with all interior angles equal to zero, and thus not a quasiconformal curve. Nevertheless, for each $n\ge 2$, $Y_{n+1}$ provides an example where $b_n=1/n$.

To support further the sharpness claim we add that there is strong numerical evidence suggesting  $\alpha_n\asymp 1/n$, $n\in\mathbb{N}$, whenever $\Gamma$ has corners. Based on such evidence, we have conjectured the strong asymptotics for non-smooth domains in \cite[pp.~520--521]{NS06}.

The first ever result regarding the strong asymptotics of $\{\lambda_n\}_{n\in\mathbb{N}}$ and $\{p_n\}_{n\in\mathbb{N}}$ was derived by Carleman in~\cite{Ca23}, for domains bounded by analytic Jordan curves. In this case the conformal map $\Phi$ has an analytic and one-to-one continuation across $\Gamma$ inside $G$.
\begin{mythm}[Carleman~\cite{Ca23}]\label{thm:carleman}
Let $L_R$ to denote the level curve  $\{z:|\Phi(z)|=R\}$ and assume that $\rho<1 $ is the smallest index for which $\Phi$ is conformal in the exterior of $L_\rho$. Then, for any $n\in\mathbb{N}$,
\begin{equation}\label{eq:carlman1}
0\le \alpha_n\le c_6(\Gamma)\, \rho^{2n}
\end{equation}
and
\begin{equation}\label{eq:carlman2}
|A_n(z)|\le c_7(\Gamma)\sqrt{n}\,\rho^n,\quad z\in\overline{\Omega}.
\end{equation}
\end{mythm}

The next major step in removing the analyticity assumption on $\Gamma$ was taken by P.K.\ Suetin in the 1960's. For his results, Suetin requires that the boundary curve $\Gamma$ belongs to a smoothness class $C(p,\alpha)$. This means that $\Gamma$ is defined by $z=g(s)$, where $s$ denotes arclength, with $g^{(p)}\in \textup{Lip}\alpha$, for some $p\in \mathbb{N}$ and $0<\alpha<1$. In this case both $\Phi$ and $\Psi:=\Phi^{-1}$ are $p$ times continuously differentiable in $\overline{\Omega}\setminus\{\infty\}$ and $\overline{\Delta}\setminus\{\infty\}$ respectively, with $\Phi^{(p)}$ and $\Psi^{(p)}$ in $\textup{Lip}\alpha$. A typical result goes as follows:
\begin{mythm}[Suetin~\cite{Su74}, Thms 1.1 \& 1.2]\label{thm:suetin}
Assume that $\Gamma\in C(p+1,\alpha)$, with $p+\alpha>1/2$. Then, for any $n\in\mathbb{N}$,
\begin{equation}\label{eq:suetin1}
0\le \alpha_n\le c_8(\Gamma)\,\frac{1}{n^{2(p+\alpha)}}
\end{equation}
and
\begin{equation}\label{eq:suetin2}
|A_n(z)|\le c_9(\Gamma)\,\frac{\log n}{n^{p+\alpha}},\quad z\in\overline{\Omega}.
\end{equation}
\end{mythm}
\begin{remark}
The results of Carleman and Suetin given above, in conjunction with Theorem~\ref{thm:alphange}, yield at once estimates for the decay of the coefficients $b_n$, depending on the degree of analyticity or smoothness of $\Gamma$. For example, under the assumptions of Suetin, for any $n\in\mathbb{N}$,
$$
|b_n|\le c_{10}(\Gamma)\,\frac{1}{n^{p+\alpha+\frac{1}{2}}}.
$$
\end{remark}

Strong asymptotics for $\lambda_n$ and $p_n$ were also derived by E.R.\ Johnston in his Ph.D.\ thesis \cite{Jo}. These asymptotics, however, were established under analytic assumptions on certain functions related with the conformal maps $\Phi$ and $\Psi$ (as compared to the geometric assumptions on $\Gamma$, in the theorems of Carleman and Suetin) and they do not provide the order of decay of the associated errors. An account of Johnston's results can be found in \cite{RW}.

Apart from the above, we cite the following representative works about strong asymptotics for complex orthogonal polynomials generated by measures supported on 2-dimensional subsets of $\mathbb{C}$:  (a) Szeg\"{o}'s book \cite[Ch.\ XVI]{Sz}, for orthogonal polynomials with respect to the arclength measure (the so-called Szeg\"{o} polynomials) on analytic Jordan curves; (b) Suetin's paper \cite{Su66b}, for weighted Szeg\"{o} polynomials on smooth Jordan curves; (c) Widom's paper \cite{Wi69} for weighted Szeg\"{o} polynomials on systems of smooth Jordan curves and smooth Jordan arcs; (d) the recent paper \cite{GPSS}, for Bergman polynomials on systems of smooth Jordan domains. This list is by no means complete.
Nevertheless, we haven't been able to trace in the literature a single result establishing strong asymptotics for orthogonal polynomials defined by measures supported on non-smooth domains, curves or arcs. In this connection, we note that the well-known approach that combines the Riemann-Hilbert reformulation of orthogonal polynomials of Fokas, Its and Kitaev \cite{FIK91,FIK92}, with the method of steepest descent, introduced by Deift and Zhou \cite{DZ}, cannot be possibly applied to treat Bergman polynomials. This is so, because this approach produces, invariably, orthogonal polynomials that satisfy a finite-term recurrence relation and this is not the case with the Bergman polynomials, as Theorem~\ref{thm:ftrr} below shows.

By contrast, strong asymptotics for orthogonal polynomials on the real line, is a well-studied subject. From the vast bibliography available, we cite the recent breakthrough paper of Lubinsky~\cite{Lupress}, on universality limits for kernel polynomials.

The paper is organized as follows: In Section~\ref{sec:faber} we study the properties of Faber polynomials and state a number of results that yield immediately the proofs of the three main theorems. The main result of Section~\ref{sec:poly-est} is a sharp estimate that relates the growth of a polynomial in $\Omega$ to its $L^2$-norm in $G$. This estimate is essential for establishing Theorem~\ref{thm:finepn}.
Section~\ref{sec:appl}  contains a variety of applications of Theorems~\ref{thm:finelambdan} and \ref{thm:finepn}. Finally, in Section~\ref{proofs} we present the proofs of various statements in Sections~\ref{sec:faber} and \ref{sec:poly-est}.

\section{Faber polynomials and proofs of the main results}\label{sec:faber}
\setcounter{equation}{0}
The Faber polynomials $\{F_n\}_{n=0}^\infty$ of $G$ are defined as the polynomial part of the expansion of $\Phi^n(z)$, $n=0,1,\ldots$, near infinity. Therefore, from (\ref{eq:Phi}),
\begin{equation}\label{eq:PhinFn}
\Phi^n(z)=F_n(z)-E_n(z),\quad z\in\Omega,
\end{equation}
where
\begin{equation}\label{eq:Fndef}
F_n(z)=\gamma^{n}z^n+\cdots,
\end{equation}
is the Faber polynomial of degree $n$ and
\begin{equation}\label{eq:Endef}
E_n(z)=\frac{c^{(n)}_1}{z}+\frac{c^{(n)}_2}{z^2}+\frac{c^{(n)}_3}{z^3}+\cdots,
\end{equation}
is the singular part of $\Phi^n(z)$. According to the asymptotics established by Carleman, the Bergman polynomial $p_n(z)$ is related to  $\Phi^n(z)\Phi^\prime(z)$. Consequently, we consider the polynomial part of $\Phi^n(z)\Phi^\prime(z)$, rather than that of $\Phi^n(z)$, and we denote the
resulting series by $\{G_n\}_{n=0}^\infty$. $G_n$ is the so-called Faber polynomial of the 2nd kind (of degree $n$) and satisfies
\begin{equation}\label{eq:PhinPhipGn}
\Phi^n(z)\Phi^\prime(z)=G_n(z)-H_n(z),\quad z\in\Omega,
\end{equation}
with
\begin{equation}\label{eq:Gndef}
G_n(z)=\gamma^{n+1}z^n+\cdots,
\end{equation}
and
\begin{equation}\label{eq:Hndef}
H_n(z)=\frac{d^{(n)}_2}{z^2}+\frac{d^{(n)}_3}{z^3}+\frac{d^{(n)}_4}{z^4}+\cdots.
\end{equation}
It follows immediately from (\ref{eq:PhinFn}) and (\ref{eq:PhinPhipGn}) that
\begin{equation}\label{eq:GnFn+1}
G_n(z)=\frac{F_{n+1}^\prime(z)}{n+1}\quad\textup{and}\quad H_n(z)=\frac{E_{n+1}^\prime(z)}{n+1}.
\end{equation}
\begin{remark}\label{rem:PhipPsip}
If $\Gamma$ is rectifiable, then $\Phi^\prime$ belongs to the Smirnov class ${E}^1(\Omega\setminus\{\infty\})$. In addition, both $\Phi^\prime$ and $\Psi^\prime$ have non-tangential limits almost everywhere on $\Gamma$ and $\partial\mathbb{D}$, respectively, and they are integrable with respect to the arclength measure; see e.g., \cite[Ch.~10]{Du70}, \cite[\S 6.3]{Po92} or \cite{KhD82}.
\end{remark}

Assume now that the boundary $\Gamma$ is rectifiable. For such boundary, Cauchy's integral formula yields the following representations for the Faber polynomials and their associated singular parts:
\begin{equation}\label{eq:FnIntRep}
F_n(z)=\frac{1}{2\pi i}\int_\Gamma\frac{\Phi^n(\zeta)}{\zeta-z}\,d\zeta,\quad z\in G,
\end{equation}
\begin{equation}\label{eq:EnIntRep}
E_n(z)=\frac{1}{2\pi i}\int_\Gamma\frac{\Phi^n(\zeta)}{\zeta-z}\,d\zeta,\quad z\in \Omega,
\end{equation}
\begin{equation}\label{eq:GnIntRep}
G_n(z)=\frac{1}{2\pi i}\int_\Gamma\frac{\Phi^n(\zeta)\Phi^\prime(\zeta)}{\zeta-z}\,d\zeta,\quad z\in G,
\end{equation}
and
\begin{equation}\label{eq:HnIntRep}
H_n(z)=\frac{1}{2\pi i}\int_\Gamma\frac{\Phi^n(\zeta)\Phi^\prime(\zeta)}{\zeta-z}\,d\zeta,\quad z\in \Omega.
\end{equation}

Next, we single out three identities, valid for every $m,n\in\mathbb{N}$, which we are going to use in the sequel:
\begin{equation}\label{eq:PhiEmHn}
\int_\Gamma H_m(z)\overline{\Phi^{n+1}}(z)\,dz=
\int_\Gamma \Phi^m(z)\Phi^\prime(z)\overline{E_{n+1}}(z)\,dz=0,
\end{equation}
and
\begin{equation}\label{eq:PhimPhin}
\frac{1}{2\pi i}\int_\Gamma \Phi^m(z)\Phi^\prime(z)\overline{\Phi^{n+1}}(z)\,dz=\delta_{m,n},
\end{equation}
where $\delta_{m,n}$ denotes Kronecker's delta function. These identities can be verified easily by using a standard argument: Replace in the integrals $\Gamma$ by some level line $L_R=\{z:|\Phi(z)|=R\}$, with $R>1$. Next, make the change of variable $w=\Phi(z)$, use the Laurent series expansion (\ref{eq:Psi}) for $\Psi(w)$, apply the residue theorem and, finally, let $R\to 1$. For the first integral in (\ref{eq:PhiEmHn}) use, in addition, the fact that $H_m(z)$ has a double zero at infinity.

With the help of $G_n$ we define what turns out to be a central factor in our investigation:
\begin{equation}\label{eq:qndef}
q_{n-1}:=G_n-\frac{\gamma^{n+1}}{\lambda_n}p_n,\quad n\in\mathbb{N}.
\end{equation}
This is a polynomial of degree at most $n-1$, but it can be identical to zero, as the special case when $G$ is a disk shows.
By noting the relation
\begin{equation}\label{eq:pnPhinPhip}
p_n(z)=\frac{\lambda_n}{\gamma^{n+1}}\,\Phi^n(z)\Phi^\prime(z)
\left\{1+\frac{H_n(z)}{\Phi^n(z)\Phi^\prime(z)}
-\frac{q_{n-1}(z)}{\Phi^n(z)\Phi^\prime(z)}\right\},
\end{equation}
which follows at once from (\ref{eq:PhinPhipGn}) and (\ref{eq:qndef}) and is valid for any $z\in\Omega$ (since $\Phi^\prime(z)\ne 0$), it is not surprising that we formulate our asymptotic results in terms of the  following two sequences of nonnegative numbers:
\begin{equation}\label{eq:betandef}
\beta_n:=\frac{n+1}{\pi}\|q_{n-1}\|^2_{L^2(G)},\quad n\in\mathbb{N},
\end{equation}
and
\begin{equation}\label{eq:epsndef}
\varepsilon_{n}:=\frac{n+1}{\pi}\int_{\Omega}|H_n(z)|^2dA(z),\quad n\in\mathbb{N}.
\end{equation}
(Recall that $H_n(z)=O(1/z^2)$, as $z\to\infty$, hence the integral in (\ref{eq:epsndef}) is finite for every $n$.)

The proof of  Theorems~\ref{thm:finelambdan} and \ref{thm:finepn} amounts, eventually, to establishing the order of decay of the two sequences $\{\beta_n\}_{n\in\mathbb{N}}$ and $\{\varepsilon_n\}_{n\in\mathbb{N}}$.
To this end, the following representation  of $\beta_n$ as a line integral will be helpful:
\begin{equation}\label{eq:betanEn}
\beta_n=\frac{1}{2\pi i}\,\int_\Gamma q_{n-1}(z)\overline{E_{n+1}}(z)\,dz.
\end{equation}
To derive (\ref{eq:betanEn}), we use the orthogonality property of $p_n$, the relations (\ref{eq:PhinFn})--(\ref{eq:GnFn+1}) and Green's formula to conclude, in steps,
\begin{align*}
\|q_{n-1}\|^2_{L^2(G)}
&=\langle q_{n-1},G_n-\frac{\gamma^{n+1}}{\lambda_n}p_n\rangle=\langle q_{n-1},G_n\rangle \\
&=\int_G q_{n-1}(z)\overline{G_{n}}(z)\,dA(z)
=\frac{1}{n+1}\int_Gq_{n-1}(z)\overline{F_{n+1}^\prime}(z)\,dA(z)\\
&=\frac{\pi}{n+1}\,\frac{1}{2\pi i}\int_\Gamma q_{n-1}(z)\overline{F_{n+1}}(z)\,dz.
\end{align*}
Hence,
$$
\frac{n+1}{\pi}\|q_{n-1}\|^2_{L^2(G)}
=\frac{1}{2\pi i}\int_\Gamma q_{n-1}(z)\overline{E_{n+1}}(z)\,dz
+\frac{1}{2\pi i}\int_\Gamma q_{n-1}(z)\overline{\Phi^{n+1}}(z)\,dz
$$
and the result then follows because the last integral vanishes, as it is easily seen by replacing $\overline{\Phi^{n+1}}(z)$ by $1/\Phi^{n+1}(z)$  and applying the residue theorem.

Using the fact that $E_{n+1}$ is analytic in $\Omega$, including $\infty$,  we obtain from (\ref{eq:epsndef}) with the help of Green's formula in the unbounded domain $\Omega$ (hence the minus sign) a line integral representation for $\varepsilon_n$ as well:
\begin{equation}\label{eq:epsnEn}
\varepsilon_{n}=-\frac{1}{2\pi i}\,\int_\Gamma H_n(z)\overline{E_{n+1}}(z)\,dz.
\end{equation}

It turns out that the strong asymptotic error for the leading coefficient $\lambda_n$ is quite simply the sum of $\beta_n$ and $\varepsilon_{n}$. (This, actually, explains the presence of the fractional term in the definition of $\beta_n$ and $\varepsilon_{n}$.)
\begin{lemma}\label{lem:alphabetaeps}
Assume that the boundary $\Gamma$ of $G$ is rectifiable. Then, for any $n\in\mathbb{N}$,
\begin{equation}\label{eq:alphabetaeps}
\frac{n+1}{\pi}\frac{\gamma^{2(n+1)}}{\lambda_n^2}=1-(\beta_n+\varepsilon_{n}).
\end{equation}
\end{lemma}
\begin{proof}
By applying Green's formula to $\|G_n\|_{L^2(G)}^2$ and using (\ref{eq:GnFn+1}) along with (\ref{eq:PhiEmHn}), (\ref{eq:PhimPhin}) and (\ref{eq:epsnEn}), it is not difficult to verify that,
\begin{equation}\label{eq:Gnepsn}
\|G_n\|_{L^2(G)}^2=\frac{\pi}{n+1}\,\left(1-\varepsilon_n\right),\quad n\in\mathbb{N}.
\end{equation}
This relation was established in \cite[p.~66]{Jo} (for slightly different normalized Faber polynomials), and was reported in \cite[p.~297]{RW}. It was also used in \cite[pp.~12--13]{Su74}.

With  (\ref{eq:Gnepsn}) at hand, the result (\ref{eq:alphabetaeps}) emerges from (\ref{eq:betandef}), because
\begin{eqnarray}\label{eq:pytha}
\|G_n\|^2_{L^2(G)}=\|\frac{\gamma^{n+1}}{\lambda_n}\,p_n+q_{n-1}\|^2_{L^2(G)}
       =\frac{\gamma^{2(n+1)}}{\lambda_n^2}+\|q_{n-1}\|^2_{L^2(G)},
\end{eqnarray}
where we made use of Pythagoras' theorem for the norm  $\|\cdot\|_{L^2(G)}$ and of the fact that $p_n$ is the orthonormal polynomial of degree $n$.
\end{proof}

The proof of Theorem~\ref{thm:finelambdan} will be an immediately consequence of Lemma~\ref{lem:alphabetaeps}, and the following two theorems.

\begin{theorem}\label{thm:betan}
Assume that $\Gamma$ is quasiconformal and rectifiable. Then, for any $n\in\mathbb{N}$,
\begin{equation}\label{eq:betan}
0\le\beta_n\le \frac{k^2}{1-k^2}\,\,\varepsilon_{n},
\end{equation}
where $k$ is a reflection factor of $\Gamma$.
\end{theorem}
Note that $\beta_n$, $\varepsilon_n$ and $k$ vanish simultaneously if $\Gamma$ is a circle.

The next theorem is established for $\Gamma$ piecewise analytic without cusps. This means that $\Gamma$ consists of a finite number of analytic arcs, say $N$, that meet at corner points $z_j$, $j=1,\ldots,N$, where they form exterior angles $\omega_j\pi$, with $0<\omega_j<2$.
\begin{theorem}\label{thm:epsn}
Assume that $\Gamma$ is piecewise analytic without cusps. Then, for any $n\in\mathbb{N}$,
\begin{equation}\label{eq:epsn}
0\le\varepsilon_{n}\le c_1(\Gamma)\,\frac{1}{n},
\end{equation}
where $c_1(\Gamma)$ depends on $\Gamma$ only.
\end{theorem}
It is interesting to note the universality aspect in the estimate (\ref{eq:epsn}), in the sense that the geometry of $\Gamma$, as it is measured by the values of the angles $\omega_j\pi$, does not influence the way that $\varepsilon_{n}$ tends to zero. This is certainly out of the ordinary in approximation theoretical results involving domains with corners, and it can be contributed to the fact that the effect of $\omega_j$'s \lq\lq cancels out" in the representation (\ref{eq:HnIntRep}) of $H_n(z)$; cf., for instance, equation (\ref{eq:Iji}) below.

\begin{proof}[Proof of Theorem~\ref{thm:finelambdan}.]
The assumption of the theorem implies that $\Gamma$ is quasiconformal and rectifiable. Hence, as it was noted above, the result (\ref{eqinthm:finelambdanii}) follows by combining Lemma~\ref{lem:alphabetaeps} with Theorems~\ref{thm:betan} and \ref{thm:epsn}.
\end{proof}

The proofs of Theorems~\ref{thm:betan} and \ref{thm:epsn} are given in Sections \ref{sub:QuasiRef} and \ref{sec:piece-anal}, respectively. Here we single out a simple consequence of the two estimates (\ref{eq:betan}) and (\ref{eq:epsn}):
\begin{corollary}\label{cor:qn-1L2}
Assume that $\Gamma$ is piecewise analytic without cusps. Then, for any $n\in\mathbb{N}$,
\begin{equation}\label{eq:qn-1L2}
\|q_{n-1}\|_{L^2(G)}\le c_2(\Gamma)\,\frac{1}{n},
\end{equation}
where $c_2(\Gamma)$ depends on $\Gamma$ only.
\end{corollary}

The relation (\ref{eq:pnPhinPhip}) shows that in order to derive the strong asymptotics for $p_n(z)$ in $\Omega$, we need suitable estimates for $q_{n-1}(z)$ and $H_n(z)$, $z\in\Omega$. For $q_{n-1}(z)$ this is provided by Lemma~\ref{lem:PolyLemma} below. Regarding $H_n(z)$, it follows easily from the representation (\ref{eq:HnIntRep}) and Remark~\ref{rem:PhipPsip} that, for $\Gamma$ rectifiable,
$$
|H_n(z)|\le\frac{c_3(\Gamma)}{\dist(z,\Gamma)},\quad z\in\Omega.
$$
This is sufficient to yield $A_n(z)=O(1/n)$, for $z$ in $\Omega$. However, for the more explicit estimate (\ref{eqinthm:finepnii1}) we need the following theorem, whose proof is given in Section~\ref{sec:piece-anal}.
\begin{theorem}\label{thm:HnOmgae}
Assume that $\Gamma$ is piecewise analytic without cusps. Then, for any $n\in\mathbb{N}$,
\begin{equation}\label{eq:HnOmega}
|H_n(z)|\le\frac{c_4(\Gamma)}{\dist(z,\Gamma)}\,\frac{1}{n},\quad z\in\Omega,
\end{equation}
where $c_4(\Gamma)$ depends on $\Gamma$ only.
\end{theorem}

\begin{proof}[Proof of Theorem~\ref{thm:finepn}.]
From Lemma~\ref{lem:PolyLemma} (with $P\equiv q_{n-1}$) and Corollary~\ref{cor:qn-1L2} we have
\begin{equation}
|q_{n-1}(z)|\le\frac{c_5(\Gamma)}{\dist(z,\Gamma)}\,\frac{1}{\sqrt{n}}\,|\Phi(z)|^{n},
\quad z\in\Omega.
\end{equation}
The result then follows easily from (\ref{eq:pnPhinPhip}) and (\ref{eq:HnOmega}), because Theorem~\ref{thm:finelambdan} implies that
\begin{equation}\label{eq:lam/gam}
\frac{\lambda_n}{\gamma^{n+1}}=\sqrt{\frac{n+1}{\pi}}\left(1+\xi_n\right),\quad n\in\mathbb{N},
\end{equation}
where
$$
0\le\xi_n\le c_6(\Gamma)\,\frac{1}{n}.
$$
\end{proof}

The last result of this section provides a lower bound for $\varepsilon_n$ and yields immediately Theorem~\ref{thm:alphange}. Its own proof is given in Section~\ref{sub:QuasiRef}.
\begin{lemma}\label{lem:epsnge}
Assume that $\Gamma$ is quasiconformal. Then, for any $n\in\mathbb{N}$,
\begin{equation}\label{eq:epsge}
\varepsilon_{n}\ge\,\frac{\pi\,(1-k^2)}{A(G)}\,(n+1)\,|b_{n+1}|^2,
\end{equation}
where $A(G)$ denotes the area of $G$ and $k$ is a reflection factor of $\Gamma$.
\end{lemma}

\begin{proof}[Proof of Theorem~\ref{thm:alphange}.]
From Lemma~\ref{lem:alphabetaeps} it follows that
$$
\alpha_n=\beta_n+\varepsilon_n.
$$
Thus, (\ref{eq:alphange}) is a straight consequence of Lemma~\ref{lem:epsnge}.
\end{proof}

\section{Polynomial Estimates}\label{sec:poly-est}
\setcounter{equation}{0}
In the proof of Theorem~\ref{thm:finepn}, we required an estimate for the growth of the polynomial $q_{n-1}$ in $\Omega$, in terms of its $L^2$-norm in $G$. This is the purpose of the next lemma, where we use $\mathbb{P}_n$  to denote the space of the polynomials of degree up to $n$.
\begin{lemma}\label{lem:PolyLemma}
Assume that $\Gamma$ is quasiconformal and rectifiable. Then, for any $P\in\mathbb{P}_n$,
\begin{equation}\label{eq:PolyLemma}
|P(z)|\le\frac{1}{\dist(z,\Gamma)\sqrt{1-k^2}}\,\,\sqrt{\frac{n+1}{\pi}}\,
\|P\|_{L^2(G)}\,|\Phi(z)|^{n+1},\quad z\in\Omega,
\end{equation}
where $k$ is a reflection factor of $\Gamma$.
\end{lemma}
Regarding sharpness of the inequality (\ref{eq:PolyLemma}), we note that the order $1/2$ of $n$ cannot be improved, as the strong asymptotics of Section~\ref{section:intro} show. Also, the constant term is asymptotically optimal for $z\to\infty$, as the case $P(z)=z^n$, with $G=\mathbb{D}$ (hence $k=0$) shows.

Lemma~\ref{eq:PolyLemma} should be compared with the following well-known result that gives the growth of a polynomial in terms of its uniform norm on $\overline{G}$, denoted here by $\|P\|_{\overline{G}}$.
\begin{mylemma}[Bernstein-Walsh]\label{lem:B-W}
For any $P\in\mathbb{P}_n$,
\begin{equation}\label{eq:B-W}
|P(z)|\le\|P\|_{\overline{G}}\,\,|\Phi(z)|^{n},\quad z\in\Omega.
\end{equation}
\end{mylemma}
We note that the inequality (\ref{eq:B-W}) is valid under more general assumption for $\overline{G}$; cf.\ e.g. \cite[p.~153]{ST}.
We also note the following norm comparison result, which was derived by Suetin in \cite[p.~38]{Su74} under the assumption $\Gamma$ is smooth:
\begin{equation}\label{eq:Suetin-ineq}
\|P\|_{\overline{G}}\le c(\Gamma)\,(n+1)\,\|P\|_{L^2(G)}.
\end{equation}
To stress the importance of Lemma~\ref{lem:PolyLemma} for our work here, we observe that the combination of (\ref{eq:B-W}) with (\ref{eq:Suetin-ineq}) yields an estimate for the growth of $P$ in $\Omega$ which is $O(n)$, rather than $O(\sqrt{n})$, and this is just not adequate for establishing the strong asymptotics for $p_n$, even for $\Gamma$ smooth; see for details the proof of Theorem~\ref{thm:finepn}.

We turn, now,  our attention to the decay of $p_n(z)$ in $G$, for $\Gamma$ piecewise analytic, and present results for $z\in B$, where $B$ is a compact subset of $G$. (Below we use $c_j(\Gamma,B)$ to denote constants that depend only on $\Gamma$ and $B$.) Before that, we note an estimate for the decay of the derivative of the Faber polynomials:
\begin{equation}\label{eq:GaierFn}
|F_{n+1}^\prime(z)|\le{c_1(\Gamma,B)}\,\frac{1}{n^\omega},\quad z\in B,
\end{equation}
where $\omega\pi$ ($0<\omega<2$) is the smallest exterior angle of $\Gamma$; see \cite[p.\ 223]{Ga01}. This, in view of (\ref{eq:GnFn+1}), gives immediately,
\begin{equation}\label{eq:Gnz}
|G_n(z)|\le {c_2(\Gamma,B)}\,\frac{1}{n^{\omega+1}},\quad z\in B.
\end{equation}
This latter inequality leads to an estimate of the decay of $p_n$ on B. Indeed, from \cite[Lem.\ 1, p.\ 4]{Ga} we have,
$$
|q_{n-1}(z)|\le\frac{\|q_{n-1}\|_{L^2(G)}}{\sqrt{\pi}\,{\dist(z,\Gamma)}},\quad z\in G,
$$
and this in conjunction with (\ref{eq:qndef}), (\ref{eq:qn-1L2}), (\ref{eq:lam/gam}) and (\ref{eq:Gnz}) leads to following:
\begin{lemma}\label{lem:pnG}
Assume that $\Gamma$ is piecewise analytic without cusps. Then, for any $n\in\mathbb{N}$,
\begin{equation}\label{eq:pnG}
|p_n(z)|\le{c_3(\Gamma,B)}\,\frac{1}{n^{{1}/{2}}},\quad z\in B.
\end{equation}
\end{lemma}
The result of Lemma~\ref{lem:pnG} should be combined with an estimate derived in \cite[p.~530]{MSS}. That is, for any $n\in\mathbb{N}$,
\begin{equation}\label{eq:pnBGaier}
|p_n(z)|\le\ c_4(\Gamma,B)\,\frac{1}{n^{s}},\quad z\in B,
\end{equation}
where ${s:=\min_{1\le j\le N}\{\omega_j/(2-\omega_j)\}}$. In particular, if $\omega_j<2/3$, for some $j$, then (\ref{eq:pnG}) is better than (\ref{eq:pnBGaier}).

Regarding sharpness of the exponent of $n$ in (\ref{eq:pnG}), we emphasize the following result of \cite[p.\ 531]{MSS}: \textit{If not all interior angles of $\Gamma$ are of the  form $\pi/m$, $m\in\mathbb{N}$, and if we disregard in the definition of $s$ angles of this form, should there exists, then for any $\epsilon >0$, there is a subsequence $\mathcal{N}_\epsilon\subset\mathbb{N}$, such that for any} $n\in\mathcal{N}_\epsilon$,
$$
|p_n(z)|\ge c_5(\Gamma,B)\,\frac{1}{n^{s+1/2+\epsilon}},\quad z\in B.
$$

\section{Applications}\label{sec:appl}
\setcounter{equation}{0}
Strong asymptotics for orthogonal polynomials with respect to measures supported on the real line play a central role in the development of the theory of orthogonal polynomials in $\mathbb{R}$.
This we expect to be the case for Bergman polynomials also. Accordingly, in this section we show how Theorems~\ref{thm:finelambdan} and \ref{thm:finepn} can be used in order to refine: (a) two classical result on the distribution of zeros and the week asymptotics of $\{p_n\}_{n\in\mathbb{N}}$; (b) two recent results on recurrence relations and the algebraicity of solutions of the Dirichlet problem.

\subsection{Zeros of the Bergman polynomials}\label{subsec:zeros}
A well-known result of Fejer asserts that \textit{all the zeros of} $\{p_n\}_{n\in\mathbb{N}}$, \textit{are contained on the convex hull} $\textup{Co}(\overline{G})$ of $\overline{G}$. This was refined by Saff~\cite{Sa90} to the interior of $\textup{Co}(\overline{G})$. To the above it should be added a result of Widom \cite{Wi67} to the effect that, \textit{on any closed subset $B$ of $\Omega\cap\textup{Co}(\overline{G})$ and for any $n\in\mathbb{N}$, the number of zeros of $p_n$ on $B$ is bounded independently of $n$}. This of course, doesn't preclude the possibility that, if $B\neq\emptyset$, $p_n$ has a zero on $B$, for every $n\in\mathbb{N}$. Our result, which is a simple consequence of Theorem~\ref{thm:finepn}, shows that, under an additional assumption on $\partial G$, the zeros of the sequence $\{p_n\}_{n\in\mathbb{N}}$ cannot be accumulated in $\Omega$.
\begin{theorem}\label{thn:zeros}
Assume that $\Gamma$ is piecewise analytic without cusps. Then, for any closed set $B\subset\Omega$, there exists $n_0\in\mathbb{N}$, such that for $n\ge n_0$,
$p_n$ has no zeros on $B$.
\end{theorem}

\subsection{Weak asymptotics}\label{subsec:weak}
The next theorem shows how an important result of Stahl and Totik \cite[Thm~3.2.1(ii)]{StTobo}, regarding the $n$-th root behavior of $\{p_n\}_{n\in\mathbb{N}}$ in $\Omega$, can be made more precise, under an additional assumption on the boundary. Its proof follows easily after utilizing Theorem~\ref{thn:zeros} into \cite[Thm III.4.7]{ST}.
\begin{theorem}
Assume that $\Gamma$ is piecewise analytic without cusps. Then,
$$
\lim_{n \to\infty}|p_n(z)|^{1/n}=|\Phi(z)|,\quad
z\in{\Omega}.
$$
\end{theorem}
For an account on the weak asymptotics for Bergman polynomials defined by more general sets we refer to \cite[Prop.~3.1]{GPSS}.

\subsection{Ratio asymptotics}
Here, we derive the ratio asymptotics as a straight consequence of Theorems~\ref{thm:finelambdan} and \ref{thm:finepn}. Thus, we are obliged to assume that $\Gamma$ is piecewise analytic without cusps. However, we believe that the result of the next two corollaries remains valid under much weaker assumptions on $\Gamma$.
\begin{corollary}\label{cor:ratioln}
Assume that $\Gamma$ is piecewise analytic without cusps. Then, for any $n\in\mathbb{N}$,
\begin{equation}\label{eq:ratioln1}
\sqrt{\frac{n+1}{n+2}}\frac{\lambda_{n+1}}{\lambda_n}=\gamma+\varsigma_n,
\end{equation}
where
\begin{equation}\label{eq:ratioln2}
|\varsigma_n|\le c_1(\Gamma)\,\frac{1}{n}.
\end{equation}
\end{corollary}
Since $\textup{cap}(\Gamma)=1/\gamma$, (\ref{eq:ratioln1}) provides the means for computing approximations to the capacity of $\Gamma$, by using only the leading coefficients of the Bergman polynomials.

\begin{corollary}\label{cor:ratiopn}
Under the assumptions of Corollary \ref{cor:ratioln}, for any $z\in\Omega$ and sufficiently large $n\in\mathbb{N}$,
\begin{equation}\label{eq:ratiopn1}
\sqrt{\frac{n+1}{n+2}}\frac{p_{n+1}(z)}{p_n(z)}=\Phi(z)\left\{1+B_n(z)\right\},
\end{equation}
where
\begin{equation}\label{eq:ratiopn2}
|B_n(z)|\le \frac{c_2(\Gamma)}{\dist(z,\Gamma)|\Phi^\prime(z)|}\,\frac{1}{\sqrt{n}}
+c_2(\Gamma)\,\frac{1}{n}.
\end{equation}
\end{corollary}

Corollary~\ref{cor:ratiopn} suggests a simple numerical method for computing approximations to the conformal map $\Phi(z)$, for $z\in\Omega$. This is quite appealing, in the sense that the Bergman polynomials alone suffice to provide approximations to both interior (via the well-known Bergman kernel method) and exterior conformal mapping associated with the same Jordan curve.

\subsection{Finite recurrence relations and Dirichlet problems}
\begin{definition}\label{def:ftrr}
We say that the  polynomials $\{p_n\}_{n=0}^\infty$ satisfy an $(M+1)$-\textit{term recurrence relation}, if for any $n\geq M-1$,
$$
zp_n(z)=a_{n+1,n}p_{n+1}(z) + a_{n, n} p_n(z) + \cdots +
a_{n-M+1, n} p_{n-M+1}(z).
$$
\end{definition}
A direct application of the ratio asymptotics for $\{p_n\}_{n\in\mathbb{N}}$, given by Corollary~\ref{cor:ratiopn}, leads to the next two theorems. These refine, respectively, Theorems~2.2 and 2.1 of \cite{KhSt}, in the sense that they weaken the $C^2$-smoothness assumption on $\Gamma$. For their proof, it is sufficient to note that: (a) the two theorems are equivalent to each other and (b) the reason for assuming that $\Gamma$ is  $C^2$-smooth in Theorem~2.2 of \cite{KhSt} was to ensure the ratio asymptotics of the Bergman polynomials; see \cite[\S 4 Rem.~(i)]{KhSt}.
\begin{theorem}\label{thm:ftrr}
Assume that $\Gamma$ is piecewise analytic without cusps. If the Bergman polynomials $\{p_n\}_{n=0}^\infty$ satisfy an $(M+1)$-term recurrence relation, with some $M\ge 2$, then $M=2$ and $\Gamma$ is an ellipse.
\end{theorem}

\begin{theorem}\label{thm:DP}
Let $G$ be a bounded simply-connected domain with Jordan boundary $\Gamma$, which is piecewise analytic without cusps. Assume that there exists a positive integer $M:=M(G)$ with the property that the Dirichlet problem
\begin{equation}\label{eq:DP}
\left\{
\begin{alignedat}{2}
&\Delta u=0\quad &&\text{in} \ \ G, \\
&u=\overline{z}^mz^n\quad&&\text{on}\ \ \Gamma,
\end{alignedat}
\right.
\end{equation}
has a polynomial solution of degree $\le m(M-1)+n$ in $z$ and of degree $\le n(M-1)+m$ in $\overline{z}$, for all positive integers $m$ and $n$. Then $\Gamma$ is an ellipse and $M=2$.
\end{theorem}
Theorem~\ref{thm:DP} confirms a special case of the so-called Khavinson and Shapiro conjecture; see \cite{KL} for results reporting on the recent progress in this direction. We note that the equivalence between
the two properties \lq\lq the Bergman polynomials of $G$ satisfy a finite-term recurrence relation" and \lq\lq any Dirichlet problem in $G$, with polynomial data, possesses a polynomial solution" was first established in \cite{PuSt}.


\section{Proofs}\label{proofs}
\setcounter{equation}{0}
\subsection{Quasiconformal boundary}\label{sub:QuasiRef}
Assume that $\Gamma$ is a quasiconformal curve. Our arguments in this subsection are based on the use of a $K$-quasiconformal reflection $y:\overline{\mathbb{C}}\to\overline{\mathbb{C}}$ defined, for some $K\ge 1$, by $\Gamma$ and a fixed point $a$ in $G$. The existence of such a reflection was established by Ahlfors in \cite{Ah63}. 
Below, we collect together some well-known properties of $y(z)$ which are important for our work and we refer to the two monographs \cite{LV} and \cite[pp.~17--27, 108--109]{ABD}, for a concise account of basic results in quasiconformal mapping theory:

\noindent
\textit{Properties of quasiconformal reflection.}
\begin{enumerate}[\qquad]\itemsep=0pt
\item[(A1)]
$\overline{y}$ is a $K$-quasiconformal mapping;
\item[(A2)]
$y(z)$ is continuously differentiable in $\mathbb{C}\setminus\{\Gamma\cup \{a\}\}$;
\item[(A3)]
$y(G)=\Omega$, $y(\Omega)=G$, $y(a)=\infty$ and $y(\infty)=a$;
\item[(A4)]
$y(z)=z$, for every $z\in\Gamma$ and $y(y(z))=z$, for all $z\in\mathbb{C}$.
\end{enumerate}
From the property (A1) it follows that $y$ is a sense-reversing homeomorphism of $\overline{\mathbb{C}}$ onto $\overline{\mathbb{C}}$, satisfying almost everywhere in $\mathbb{C}$ that
\begin{equation}\label{eq:yzk}
\left|\frac{y_z}{y_{\overline{z}}}\right|=\left|\frac{\overline{y}_{\overline{z}}}{\overline{y}_{{z}}}\right|
\le k=\frac{K-1}{K+1}<1,
\end{equation}
where $y_z$ and $y_{\overline{z}}$ denote the formal (complex) derivatives of $y$ with respect to $z$ and $\overline{z}$ and $k$ is a reflection factor of $\Gamma$ with respect to $y$. Furthermore, it follows that $y(z)$ has $L^2$-derivatives in $\mathbb{C}$, in the sense that $y(z)$ is absolutely continuous on lines in $\mathbb{C}$ and belongs to the Sobolev space $W^{1,2}_{loc}$. Clearly, from (\ref{eq:yzk}), almost everywhere in $\mathbb{C}$,
\begin{equation}\label{eq:yzJ}
|y_{\overline{z}}|^2\le\frac{-1}{1-k^2}\,J(y(z))\quad\textup{and}\quad
|y_z|^2\le\frac{-k^2}{1-k^2}\,J(y(z)),
\end{equation}
where $J(y(z)):=|y_z|^2-|y_{\overline{z}}|^2\,\,(<0)$ is the Jacobian of the transformation $y:\overline{\mathbb{C}}\to\overline{\mathbb{C}}$. These two inequalities yield immediately, in view of (A3), that
\begin{equation}\label{eq:intyzJ}
\int_\Omega|y_{\overline{z}}|^2dA(z)\le\frac{1}{1-k^2}A(G)\quad\textup{and}\quad
\int_\Omega|y_z|^2dA(z)\le\frac{k^2}{1-k^2}A(G).
\end{equation}

\begin{proof}[Proof of Theorem~\ref{thm:betan}]
From the properties of the quasiconformal reflection (A1)--(A4), in conjunction with the expansion (\ref{eq:Endef}), we see that $E_{n+1}(y(z))$ defines a continuous extension of $E_{n+1}(z)$ onto $\overline{G}$, which has $L^2$-derivatives. Since $y(z)=z$, for $z\in \Gamma$, (\ref{eq:betanEn}) can be written as
$$
\beta_n=\frac{1}{2\pi i}\,\int_\Gamma q_{n-1}(z)\,(\overline{E_{n+1}}\circ y)(z)\,dz.
$$
Hence, by using Green's formula we obtain, 
\begin{align*}
\beta_n
&=\frac{1}{\pi}\,\int_G \left[q_{n-1}(z)\,(\overline{E_{n+1}}\circ y)(z)\right]_{\overline{z}}\,dA(z)\\
&=\frac{1}{\pi}\,\int_G q_{n-1}(z)\,\overline{E_{n+1}^\prime}(y(z))\,\overline{y}_{\overline{z}}\,\,dA(z).
\end{align*}
Our task now is to find a suitable upper bound for the area integral above. For this, we use the Cauchy-Schwarz inequality and  (\ref{eq:yzJ}) to obtain
\begin{align*}\label{eq:betanCS}
\beta_n &\le\frac{1}{\pi}\,\|q_{n-1}\|_{L^2(G)}
\left[\frac{-k^2}{1-k^2}\,\int_G \left|{E_{n+1}^\prime}(y(z))\right|^2 J(y(z))\,dA(z)\right]^{1/2}\\
&=\frac{1}{\pi}\,\|q_{n-1}\|_{L^2(G)}
\left[\frac{k^2}{1-k^2}\,\int_\Omega \left|{E_{n+1}^\prime}(\zeta)\right|^2 \,dA(\zeta)\right]^{1/2}\\
&=\frac{n+1}{\pi}\,\|q_{n-1}\|_{L^2(G)}\left[\frac{k^2}{1-k^2}\,\int_\Omega \left|H_{n}(z)\right|^2 \,dA(z)\right]^{1/2},
\end{align*}
where we made use of $y(\Omega)=G$ and (\ref{eq:GnFn+1}). The result (\ref{eq:betan}) then follows from the definition of $\beta_n$ and $\varepsilon_n$ in (\ref{eq:betandef}) and (\ref{eq:epsndef}).
\end{proof}

\begin{proof}[Proof of Lemma~\ref{lem:epsnge}]
The application of the residue theorem to the integral of $\Phi^{n+1}(z)$ over $L_R$, for some sufficiently  large $R$, followed by the change of variable $z=\Psi(w)$, in conjunction with the expansions (\ref{eq:Endef}) and (\ref{eq:Psi}), shows that, for any $n\in\mathbb{N}$,
\begin{equation*}
c_1^{(n+1)}=(n+1)b_{n+1}.
\end{equation*}
This, in view of (\ref{eq:Endef}) and (\ref{eq:GnFn+1}) shows that near infinity,
$$
H_n(z)=-\frac{b_{n+1}}{z^2}+\frac{d_3^{(n+1)}}{z^3}+\cdots,
$$
and the application of \cite[Thm 8, \S 6]{Be77} (see also \cite[Lem.\ 2.4.2]{ABD}) leads to the expression
\begin{equation}\label{eq:bn+1Hn}
b_{n+1}=\frac{1}{\pi}\int_\Omega H_n(z) y_{\overline{z}}\,dA(z).
\end{equation}
Once more, the Cauchy-Schwarz inequality, followed by the first inequality in (\ref{eq:intyzJ}), yields
\begin{align*}
\left|\int_\Omega H_n(z) y_{\overline{z}}\,dA(z)\right|
&\le \left[\int_\Omega|H_n(z)|^2dA(z) \right]^{1/2}\left[\int_\Omega|y_{\overline{z}}|^2dA(z) \right]^{1/2}\\
&\le \left[\int_\Omega|H_n(z)|^2dA(z)\right]^{1/2} \left[\frac{1}{1-k^2}A(G)\right]^{1/2},
\end{align*}
and the required result emerges from (\ref{eq:bn+1Hn}) and the definition of $\varepsilon_n$.
\end{proof}

\begin{proof}[Proof of Lemma~\ref{lem:PolyLemma}]
Let $P\in\mathbb{P}_n$ and fix $z\in\Omega$. Then, the function ${P}/{\Phi^{n+1}}$ is analytic in $\Omega$, continuous on $\Gamma=\partial\Omega$ and vanishes at $\infty$. Hence, from Cauchy's formula and the identity property (A4) of $y(z)$ we have,
\begin{equation*}
\frac{P(z)}{\Phi^{n+1}(z)}
=-\frac{1}{2\pi i}\int_\Gamma\frac{g(\zeta)\,d\zeta}{\Phi^{n+1}(\zeta)}
=-\frac{1}{2\pi i}\int_\Gamma\frac{g(\zeta)\,d\zeta}{(\Phi^{n+1}\circ y)(\zeta)},
\end{equation*}
where $g(\zeta):=P(\zeta)/(\zeta-z)$. Now, the function $1/\Phi^{n+1}\circ y$ is continuous on $\overline{G}$, and has $L^2$-derivatives in $G$. Hence, Green's formula yields
\begin{align}\label{eq:P/Phin}
\frac{P(z)}{\Phi^{n+1}(z)} &= -\frac{1}{\pi}\,\int_G
\left[\frac{g(\zeta)}{(\Phi^{n+1}\circ y)(\zeta)}\right]_{\overline{\zeta}}\,dA(\zeta)\nonumber \\
&= \frac{n+1}{\pi}\,\int_G
g(\zeta)\,\frac{\Phi^\prime(y(\zeta))}{(\Phi^{n+2}\circ y)(\zeta)}\,y_{\overline{\zeta}}\,\,dA(\zeta),
\end{align}
where we made use of the fact that $g$ is analytic on $\overline{G}$. Next, using (\ref{eq:yzJ}) we have:
\begin{align}\label{eq:Phip/Phin}
\int_G\frac{|\Phi^\prime(y(\zeta))|^2\,|y_{\overline{\zeta}}|^2}{|(\Phi^{n+2}\circ y)(\zeta)|^2}\, dA(\zeta)
&\le \frac{-1}{1-k^2}\int_G\frac{|\Phi^\prime(y(\zeta))|^2\,|J(y(\zeta))|^2}{|(\Phi^{n+2}\circ y)(\zeta)|^2}\,
 dA(\zeta)\nonumber \\
&=\frac{1}{1-k^2}\int_\Omega\frac{|\Phi^\prime(t)|^2}{|\Phi^{n+2}(t)|^2}\,dA(t)\nonumber \\
&=\frac{1}{1-k^2}\int_\Delta\frac{dA(w)}{|w^{n+2}|^2}
=\frac{1}{1-k^2}\frac{\pi}{(n+1)}.
\end{align}
Obviously,
\begin{equation*}
\int_G|g(\zeta)|^2dA(\zeta)
\le\frac{\|P\|_{L^2(G)}^2}{(\dist(z,\Gamma))^2},
\end{equation*}
and the result (\ref{eq:PolyLemma}) follows from the application of the Cauchy-Schwarz inequality to the integral in (\ref{eq:P/Phin}).
\end{proof}

\subsection{Piecewise analytic boundary}\label{sec:piece-anal}
We recall our assumption that $\Gamma$ consists of $N$ analytic arcs, that meet at corner points $z_j$, $j=1,\ldots,N$, where they form exterior angles $\omega_j\pi$, with $0<\omega_j<2$.

The basic idea underlying the work in this subsection is simple. Extend, using Schwarz reflection, $\Phi$ across each arc of $\Gamma$ inside $G$, so that this extension is conformal in the exterior of a piecewise analytic Jordan curve $\Gamma^\prime$, which shares with $\Gamma$ the same corners $z_j$ and otherwise lies in $G$. $\Gamma^\prime$ can be chosen so that $\Phi$ is analytic on $\Gamma^\prime$, apart from the corners $z_j$. (Hence, the four representations (\ref{eq:FnIntRep})--(\ref{eq:HnIntRep}) remains valid if $\Gamma$ is deformed to $\Gamma^\prime$.) Next, divide  $\Gamma^\prime$ into two parts: a part $l$ containing arcs emanating from the corners $z_j$, and a part $\tau$, containing the remainder of $\Gamma^\prime$, so that there exists a compact set $B:=B(\Gamma)$ of $G$  which includes $\tau$. When $\zeta\in\tau$, $\Phi(\zeta)^n$ decays geometrically to zero, i.e. $|\Phi(\zeta)|^n=O(\rho^n)$, for some $\rho:=\rho(B)<1,$ and therefore its contribution is negligible, when compared with the contribution of $\Phi(\zeta)^n$, for $\zeta\in l$.
To make things more precise, we assume (as we may) that $l$ is formed by linear segments, and we number those two meeting at $z_j$ by $l_j^i$, $i=1,2$; see Figure~\ref{fig:Gammapr}.
\begin{figure}[h]
\begin{center}
\includegraphics*[scale=0.7]{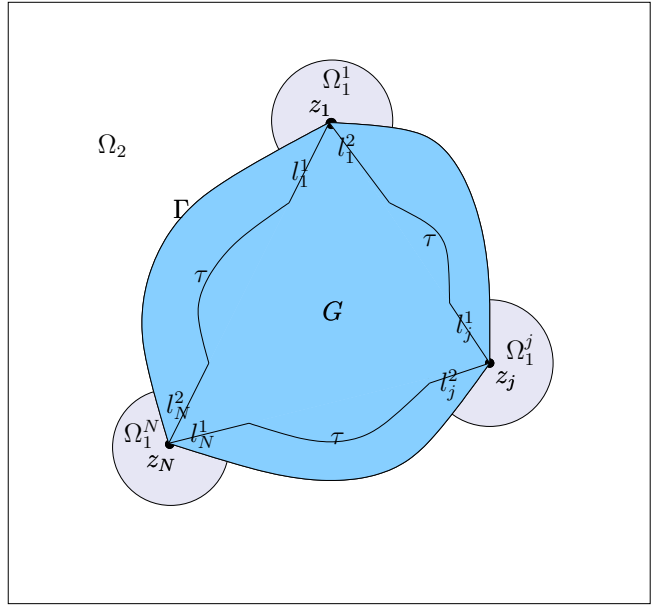}
\caption{The two decompositions: $\Gamma^\prime=l\cup\tau$ and $\Omega=\Omega_1\cup\Omega_2$.}
\label{fig:Gammapr}
\end{center}
\end{figure}

In the sequel, we make extensive use of the following four inequalities:
\begin{remark}\label{rem:Lehman}
For any $\zeta\in l_j^i$ we have:
\begin{enumerate}\itemsep=2pt
\item[\textup{(i)}]
$\displaystyle{|\Phi(\zeta)-\Phi(z_j)|\ge c\,|\zeta-z_j|^{1/\omega_j}}$;
\item[\textup{(ii)}]
$\displaystyle{|\Phi^\prime(\zeta)|\le c\,|\zeta-z_j|^{1/\omega_j-1}}$;
\item[\textup{(iii)}]
$\displaystyle{|\Phi(\zeta)|\le 1-c\,|\zeta-z_j|^{1/\omega_j}}$;
\item[\textup{(iv)}]
$\displaystyle{\dist(\zeta,\Gamma)\ge c\,|\zeta-z_j|}$.
\end{enumerate}
\end{remark}
(Above and in the remainder of the paper, we use the symbol $c$ generically in order to denote positive constants, possibly different ones, that depend on $\Gamma$ only.)

The inequalities (i) and (ii) emerge from Lehman's asymptotic expansions of conformal mappings near an analytic corner \cite{Lehman}. 
The third inequality follows from (i), because reflection preserves angles. Finally, (iv) is a simple fact of conformal mapping geometry.

\begin{proof}[Proof of Theorem~\ref{thm:HnOmgae}]
The proof goes along similar lines as those taken in \cite{Ga01} for establishing an estimate for $F_n(z)$ of the form (\ref{eq:GaierFn}), with one significant difference: Here, $z$ lies in $\Omega$, instead of $G$, and thus can tend to $\Gamma$ without need altering the curve $\Gamma^\prime$. As a consequence, the set $B$ defined above, does not depend on $z$, and thus  $\dist(z,\tau)\ge\dist(z,B)>\dist(\Gamma,B)=c(\Gamma)$.

The details are as follows: From the discussion above we see that, for $z\in\Omega$,
\begin{align}\label{eq:Hnparts}
H_n(z)&=\frac{1}{2\pi i}
\int_{\Gamma^\prime}\frac{\Phi^n(\zeta)\Phi^\prime(\zeta)}{\zeta-z}\,d\zeta\nonumber\\
  &=\frac{1}{2\pi i}\sum_j\int_{l_j^1\cup l_j^2}
    \frac{\Phi^n(\zeta)\Phi^\prime(\zeta)}{\zeta-z}\,d\zeta
     +\frac{1}{2\pi i}\int_{\tau}\frac{\Phi^n(\zeta)\Phi^\prime(\zeta)}{\zeta-z}\,d\zeta \nonumber\\
  &=\frac{1}{2\pi i}\sum_j\int_{l_j^1\cup l_j^2}\frac{\Phi^n(\zeta)\Phi^\prime(\zeta)}{\zeta-z}\,d\zeta
     +O(\rho^n),
\end{align}
for some $\rho<1$, independent of $z$. Hence, we only need to estimate the integral
$$
I_j^i:=\int_{l_j^i}\frac{\Phi^n(\zeta)\Phi^\prime(\zeta)}{\zeta-z}\,d\zeta.
$$

Let $s$ denote the arclength on ${l_j^i}$ measured from $z_j$. Then, Remark~\ref{rem:Lehman} yields the following two inequalities. For any $\zeta\in l_j^i$,
\begin{equation}\label{eq:InePhi}
|\Phi(\zeta)|\le 1-cs^{1/\omega_j}<\exp(-cs^{1/\omega_j})\quad\textup{and}\quad
|\Phi^\prime(\zeta)|\le cs^{1/\omega_j-1}.
\end{equation}
Since $1/\omega_j>1/2$, these imply
\begin{equation}\label{eq:Iji}
|I_j^i|\le\frac{c}{\dist(z,\Gamma)}\int_0^\infty \textup{e}^{-cns^{1/\omega_j}}s^{1/\omega_j-1}\,ds=\frac{c\,\omega_j}{\dist(z,\Gamma)}\,\frac{1}{n},
\end{equation}
and the required estimate follows from (\ref{eq:Hnparts}).
\end{proof}

The next result is needed in establishing Theorem~\ref{thm:epsn}.
\begin{lemma}\label{lem:final}
With $\omega\in(0,2]$ and $k\in\mathbb{N}$, set $\delta:=k^{-\omega}$ and let
\begin{equation}\label{eq:Iome-del-k}
I(\omega,k):=\int_0^{\delta}\left[\int_r^\infty
{\textup{e}^{-ks^{1/\omega}}s^{1/\omega-2}}\,ds \right]^2rdr.
\end{equation}
Then,
\begin{equation}\label{eq:lemfinal}
I(\omega,k)\le\frac{c}{k^2}.
\end{equation}
\end{lemma}
(In the statement and proof of Lemma~\ref{lem:final}, the positive constants $c$ depend on $\omega$ only.)
\begin{proof}
We consider separately the four complementary cases: (I) $\omega=1$; (II) $0<\omega<1$; (III) $\omega=2$; (IV) $1<\omega<2$.

\medskip
\noindent
\textit{Case} (I). Note,
$$
I(1,k)=\int_0^{\delta}\left[\int_r^\infty\frac{\textup{e}^{-ks}}{s}\,ds \right]^2rdr
=\int_0^{\delta}E_1^2(kr)\,rdr,
$$
where $E_1(x)$, denotes the exponential integral $E_1(x):=\int_x^\infty{t^{-1}}{\textup{e}^{-t}}\,dt$,
with $x>0$. Using the formula $\int_0^\infty E_1^2(t)dt=2\log 2$, we thus have
$$
I(1,k)\le\delta\int_0^{\delta}E_1^2(kr)dr\le\frac{\delta}{k}\int_0^\infty E_1^2(t)dt=\frac{c}{k^2}.
$$

\medskip
\noindent
\textit{Case} (II). Now $1/\omega>1$. Consequently, for $r>0$,
$$
\int_r^\infty\textup{e}^{-ks^{1/\omega}}s^{1/\omega-2}\,ds\le
\int_0^\infty\textup{e}^{-ks^{1/\omega}}s^{1/\omega-2}\,ds=\omega\,\Gamma(1-\omega)\,k^{\omega-1},
$$
where $\Gamma(x):=\int_0^\infty t^{x-1}\textup{e}^{-t}dt$ denotes the Gamma function with argument $x>0$.
This yields
\begin{equation}\label{eq:I-CaseII}
I(\omega,k)\le c\,\frac{\delta^2}{k^{2(1-\omega)}}=\frac{c}{k^2}.
\end{equation}

\medskip
\noindent
\textit{Case} (III). We note first the formula, valid for $r>0$,
$$
\int_r^\infty\textup{e}^{-ks^{1/2}}s^{-3/2}\,ds=
2\left({\textup{e}^{-k\,r^{1/2}}}{r^{-1/2}}-k\,E_1(k\,r^{1/2})\right).
$$
Therefore,
$$
I(2,k)< c\,\int_0^\infty\textup{e}^{-2k\,r^{1/2}}dr
+ c\,k^2\int_0^\infty E_1^2(k\,r^{1/2})\,rdr=\frac{c}{k^2}+k^2\frac{c}{k^4}=\frac{c}{k^2}.
$$

\medskip
\noindent
\textit{Case} (IV). The result for $1<\omega<2$ can be established as a special of $\omega=1$ and $\omega=2$. To see this, let $h(\omega,s):={\textup{e}^{-ks^{1/\omega}}s^{1/\omega-2}}$ and split the integral from $r$ to $\infty$ in (\ref{eq:Iome-del-k}) into three parts:
\begin{equation*}
\int_r^\infty h(\omega,s)\,ds
=\int_r^\delta h(\omega,s)\,ds+\int_\delta^1 h(\omega,s)\,ds+\int_1^\infty h(\omega,s)\,ds.
\end{equation*}
Next, observe that if $s\in(0,\delta)\cup(1,\infty)$, then $h(\omega,s)$ is an increasing function of $\omega$,
hence $h(\omega,s)\le h(2,s)$. At the other hand, when $s\in(\delta,1)$, then $h(\omega,s)$ is a decreasing function of $\omega$, thus $h(\omega,s)\le h(1,s)$. These give,
$$
\int_r^\infty h(\omega,s)\,ds
\le\int_r^\infty \textup{e}^{-ks^{1/2}}s^{-3/2}\,ds+\int_r^\infty \frac{\textup{e}^{-ks}}{s}\,ds,
$$
and the result (\ref{eq:lemfinal}) follows easily by means of the estimates given Cases (I) and (III).
\end{proof}

\begin{proof}[Proof of Theorem~\ref{thm:epsn}]
We choose positive quantities
\begin{equation}\label{eq:deltaj}
\delta_j=\delta_{n,j}:=c\,{n^{-\omega_j}},\quad j=1,\ldots,N,
\end{equation}
where $c$ is small enough so that any two of the $N$ domains $\Omega_1^{j}:=\{z\in\Omega:|z-z_j|<\delta_j\}$, are disjoint from each other. Next, we split $\Omega$ into two parts $\Omega_{1}$ and $\Omega_{2}$, where
$\Omega_{1}:=\cup_j\Omega_1^{j}$; see Figure~\ref{fig:Gammapr}.

Using the partition of $\Omega$  into $\Omega_1$ and $\Omega_2$, and in view of (\ref{eq:GnFn+1}), we can express $\varepsilon_{n}$ as the sum
\begin{align}\label{eq:I1+I2}
\varepsilon_{n}&=\frac{n+1}{\pi}\int_{\Omega_{1}}|H_n(z)|^2dA(z)
+\frac{1}{\pi(n+1)}\int_{\Omega_{2}}|E^\prime_{n+1}(z)|^2dA(z)\nonumber\\
&=:J_1(n)+J_2(n).
\end{align}
Our task now is to show that both $J_1(n)$ and $J_2(n)$ are $O(1/n)$.

\noindent
\textit{Estimating $J_1(n)$.} Note that $J_1(n)\le c\,n\sum_jT_j(n)$, where
\begin{equation*}
T_j(n):=\int_{\Omega_{1}^j}|H_n(z)|^2dA(z),\quad j=1,\ldots,N.
\end{equation*}
With $z\in\Omega_1^j$, set $r:=|z-z_j|$ and observe that, in view of Remark~\ref{rem:Lehman} (iv), $|\zeta-z|\sim s+r$, if $\zeta\in l_j^1\cup l_j^2$, while $|\zeta-z|\ge c$, if $\zeta\in l_k^1\cup l_k^2$, with $k\ne j$, where, as above, $s$ denote the arclength on ${l_j^i}$ measured from $z_j$. Consequently, since $\Omega_1^j\subset\{z:|z-z_j|<\delta_j\}$, we have from (\ref{eq:Hnparts}) and (\ref{eq:InePhi}):
\begin{align*}\label{eq:defJ1+J2+J3}
T_j(n)
&\le
c\,\int_0^{\delta_j}\left[\int_0^\infty\frac{\textup{e}^{-cns^{1/\omega_j}}s^{1/\omega_j-1}}{r+s}ds
+\sum_{k\ne j}\int_0^\infty\textup{e}^{-cns^{1/\omega_k}}s^{1/\omega_k-1}\,ds\right]^2rdr \\
&\le
c\,\int_0^{\delta_j}\left[\int_0^\infty\frac{\textup{e}^{-cns^{1/\omega_j}}s^{1/\omega_j-1}}{r+s}ds
\right]^2rdr+\frac{c}{n^2}\int_0^{\delta_j}rdr\\
&\le
c\,\int_0^{\delta_j}\left[\int_0^r\frac{\textup{e}^{-cns^{1/\omega_j}}s^{1/\omega_j-1}}{r}\,ds \right]^2rdr\\
&\quad+c\,\int_0^{\delta_j}\left[\int_r^\infty{\textup{e}^{-cns^{1/\omega_j}}s^{1/\omega_j-2}}\,ds \right]^2rdr+\frac{c}{n^{2(1+\omega_j)}}.
\end{align*}
Now, we use the estimate
$$
\int_0^r\textup{e}^{-cns^{1/\omega_j}}s^{1/\omega_j-1}\,ds=\frac{c}{n}(1-\textup{e}^{-cnr^{1/\omega_j}})
<c\,r^{1/\omega_j},
$$
and employ Lemma~\ref{lem:final} to conclude that $T_j(n)\le c/n^2$,  which yields the required inequality
\begin{equation}\label{eq:I1le1/n}
J_1(n)\le\frac{c}{n}.
\end{equation}

\medskip
\noindent
\textit{Estimating $J_2(n)$.}
By using Cauchy's integral formula for the derivative of $E_{n+1}$ and arguing as in the proof of Theorem~\ref{thm:HnOmgae} we have  for $z\in\Omega$,
\begin{eqnarray}\label{eq:En+1Om2}
E^\prime_{n+1}(z)&=&\frac{1}{2\pi i}
   \int_{\Gamma^\prime}\frac{\Phi^{n+1}(\zeta)}{(\zeta-z)^2}\,d\zeta\nonumber\\
   &=& \frac{1}{2\pi i}\sum_j\int_{l_j^1\cup l_j^2}\frac{\Phi^{n+1}(\zeta)}{(\zeta-z)^2}\,d\zeta
     +O(\rho^n),
\end{eqnarray}
for some $\rho<1$, independent of $z$.

Now, when $z\in\Omega_2$ and $\zeta\in l_j^1\cup l_j^2$, $j=1,\ldots,N$, the triangle inequality and Remark~\ref{rem:Lehman} (iv) imply $|\zeta-z|\ge c\,|z-z_j|$.
Therefore, by using (\ref{eq:InePhi}) and working as above we have,
\begin{eqnarray*}
\left|\int_{l_j^i}\frac{\Phi^{n+1}(\zeta)}{(\zeta-z)^2}\,d\zeta\right|
\le\frac{c}{|z-z_j|^2}\int_0^\infty \textup{e}^{-cns^{1/\omega_j}}ds
=\frac{c\,\Gamma(\omega_j)}{|z-z_j|^2}\,\frac{1}{n^{\omega_j}}.
\end{eqnarray*}
 This, in conjunction with (\ref{eq:En+1Om2}), leads to
$$
\int_{\Omega_2}|E^\prime_{n+1}(z)|^2dA(z)
\le c\,\sum_j\frac{1}{n^{2\omega_j}}\int_{\Omega_2}\frac{dA(z)}{|z-z_j|^4}.
$$
Furthermore, since $\Omega_2\subset\{z:|z-z_j|\ge\delta_j\}$, we have from (\ref{eq:deltaj}),
\begin{align}\label{eq:cannotdobetter}
\int_{\Omega_{2}}|E^\prime_{n+1}(z)|^2dA(z)
&\le c\sum_j\frac{1}{n^{2\omega_j}}\int_{|z-z_j|>\delta_j}\frac{dA(z)}{|z-z_j|^4}\nonumber\\
&=c\sum_j\frac{1}{n^{2\omega_j}\,\delta_j^2}=c,
\end{align}
and this yields the required estimate
\begin{equation}\label{eq:I2le1/n}
I_2(n)\le\frac{c}{n}.
\end{equation}
\end{proof}
We end, by remarking that any choice for $\delta_j$, other that (\ref{eq:deltaj}), will result to weaker estimates for $\varepsilon_n$, as a comparison of (\ref{eq:I-CaseII}) with (\ref{eq:cannotdobetter}) shows.

\def\cprime{$'$}
\providecommand{\bysame}{\leavevmode\hbox to3em{\hrulefill}\thinspace}
\providecommand{\MR}{\relax\ifhmode\unskip\space\fi MR }
\providecommand{\MRhref}[2]{%
  \href{http://www.ams.org/mathscinet-getitem?mr=#1}{#2}
}
\providecommand{\href}[2]{#2}

\end{document}